\documentclass[12pt,letterpaper]{amsart}

\usepackage{tikz}
\tikzstyle{every node}=[circle, draw, fill=black!50,
                        inner sep=0pt, minimum width=3pt]

\usepackage{amsfonts,amsmath,latexsym,color,epsfig,hyperref,amssymb,bbm}
\usepackage{enumerate}

\newtheorem{theorem}{Theorem}
\newtheorem{lemma}{Lemma}

\newtheorem{prop}[lemma]{Proposition}
\newtheorem{obs}[lemma]{Observation}

\renewcommand{\le}{\leqslant}
\renewcommand{\ge}{\geqslant}

\renewcommand{\geq}{\geqslant}

\newcommand{\parag}[1]{\vspace{2mm}

\noindent{\bf #1} }

\def\qed{\ifvmode\mbox{ }\else\unskip\fi\hskip 1em plus 10fill$\Box$}

\input{epsf}

\makeatletter
\def\Ddots{\mathinner{\mkern1mu\raise\p@
\vbox{\kern7\p@\hbox{.}}\mkern2mu
\raise4\p@\hbox{.}\mkern2mu\raise7\p@\hbox{.}\mkern1mu}}
\makeatother

\def\R{\mathbb R}
\def\Z{\mathbb Z}

\def\C{\mathbb C}
\def\F{\mathbb F}

\def\N{\mathbb N}
\def\E{\mathbb E}

\title{\vspace{-0.7cm}Ruzsa's problem on Bi-Sidon sets}
\author{J\'anos Pach \and Dmitrii Zakharov}
\date{}

\begin{document}
\maketitle

\begin{abstract}
A subset $S$ of real numbers is called \emph{bi-Sidon} if it is a Sidon set with respect to both addition and multiplication, i.e., if all pairwise sums and all pairwise products of elements of $S$ are distinct. Imre Ruzsa asked the following question: What is the maximum number $f(N)$ such that every set $S$ of $N$ real numbers contains a bi-Sidon subset of size at least $f(N)$? He proved that $f(N)\geq cN^{\frac13}$, for a constant $c>0$. In this note, we improve this bound to $N^{\frac13+\frac7{78}+o(1)}$. 
\end{abstract}

\section{Introduction}

A set of elements $S$ in an abelian group $(G,+)$ is called {\it Sidon} if for any elements $a, b, a', b' \in S$ we have $a+b=a'+b'$ if and only if $\{a, b\} = \{a', b'\}$. Sidon sets are classical objects in additive combinatorics and have been intensely studied over the decades \cite{Sidon, TaoVu}. It was first proved by Singer~\cite{Singer} in 1938, that if $p$ is a power of a prime, then one can find $p+1$ integers such that their pairwise sums are all different modulo $p^2+p+1$. Since the ratio of consecutive primes tends to $1$, this immediately implies that $\{1,\ldots,N\}$ has an (additive) Sidon subset of size at least $(1-o(1))N^{1/2}$. Unaware of Singer's result, three years later, Erd\H os and Tur\'an~\cite{ErdosTuran} established a slightly weaker lower bound and an upper bound of $N^{1/2}+O(N^{1/4}).$ (For minor improvements, see~\cite{Lindstrom, Cilleruelo, Balogh}.) Erd\H os repeatedly offered \$500 for a proof or disproof of the upper bound $N^{1/2}+o(N^{1/4}),$ but the problem is still open.
\smallskip

What happens if we replace the interval $\{1, \ldots, N\}$ with an arbitrary set $A \subset \R$ of size $N$? What is the size of the largest Sidon subset that can be found in any set $A$ of $N$ reals? Intuitively, the interval $\{1, \ldots, N\}$ has the richest ``additive structure'' among all sets of size $N$, so it should be the hardest to find Sidon subsets in it. This is indeed the case: Koml\'os, Sulyok and Szemer\'edi \cite{Komlos} proved that that any finite set $A \subset \R$ contains a Sidon subset $S$ of size at least $c |A|^{1/2}$. The simplest proof of this fact uses the so called ``Ruzsa projection trick''~\cite{RuzsaProjection}, which allows us to ``embed'' any set $A$ into an arithmetic progression of comparable length while preserving all relevant linear relations between the elements in $A$. In this context, it does not really matter, whether we work over $\Z, \R,$ or $\C$. To make our note self-contained, in the next section, we include the short proofs of some well known facts that we need. 
\smallskip

Similar statements hold for \emph{multiplicative Sidon sets} $S\subset \R$, i.e., where all solutions of the equation $ab=a'b'$ in $S$ satisfy $\{a,b\}=\{a',b'\}$. We obtain that very subset $A \subset \R$ contains a multiplicative Sidon subset of size $c |A|^{1/2}$, for some $c>0$. The order of magnitude of this bound cannot be improved, as is shown by a geometric progression of length $N$. 
\smallskip

More interesting things start to happen if we try to combine the additive and multiplicative structures of the real numbers. The exciting \emph{sum-product phenomenon} in Additive Combinatorics postulates, roughly speaking, that finite subsets of real numbers cannot have very rich additive and multiplicative structures at the same time. For any $A\subset\R$, let $A+A=\{a+b\,|\;a,b\in A\}, \; A\cdot A=\{a\cdot b\;|\;a,b\in A\}$. According to the celebrated Erd\H os--Szemer\'edi \emph{sum-product conjecture}~\cite{ErSz}, for any finite $A \subset \R$, we have 
\[
\max\{|A+A|, |A \cdot A|\} \ge |A|^{2-o(1)}
\]
as $|A|\rightarrow\infty.$ There has been a huge amount of work on the sum-product-type questions in many different settings, see, e.g., \cite{Elekes, Solymosi, RudnevStevens}.
\smallskip

In what follows, the letter $c$'s appearing in different formulas will denote usually unrelated positive constants.

For any $A\subset\R,$ let $S_{+}(A)$ and $S_{\times}(A)\subset A$ denote a maximum size \emph{additive} Sidon subset and a maximum size \emph{multiplicative} Sidon subset of $A$, respectively. We have 
$$\min(|S_{+}(A)|, |S_{\times}(A)|)\ge c|A|^{1/2},$$
for a suitable constant $c>0$. Another, less known example of the sum-product phenomenon is a recent conjecture of
Klurman and Pohoata \cite{Pohoata}, according to which any finite $A\subset \R$ satisfies
\begin{equation}\label{eq:KP}
    \max(|S_{+}(A)|, |S_{\times}(A)|)\ge c|A|^{1/2+\varepsilon},
\end{equation}
for a suitable absolute constant $\varepsilon>0$. Roche-Newton and Warren \cite{Roche21} proved that if this conjecture is true, then we must have $\varepsilon \le 1/6$. Shkredov  \cite{Shkredov} proved a version of (\ref{eq:KP}) where Sidon sets are replaced with a weaker notion of $B_2[g]$ sets, see also \cite{Mudgal}.
\smallskip

In this note, we consider a different way to combine addition and multiplication for Sidon sets, which was suggested by Ruzsa \cite{RuzsaSidon}. 
He called a set $S \subset \R$ is {\em bi-Sidon} if $S$ is both additive Sidon and multiplicative Sidon, i.e., for any $a, b, a', b' \in S$ any one of the equations $a+b=a'+b'$ or $a b = a' b'$ implies that $\{a, b\} =\{a', b'\}$. Let $S_{bi}(A)$ denote a maximum size bi-Sidon subset of $A$.
Ruzsa observed \cite{Ruzsa} that for every $A \subset \R$, we have 
$|S_{bi}(A)| \ge c|A|^{1/3}$, for an absolute constant $c>0$,
and asked if this bound can be improved. Consider a couple of examples:
\begin{itemize}
    \item Let $A = \{1, \ldots, N\}$. Then there is an additive Sidon subset $S_+\subset A$ of size at least $N^{1/2}$. On the other hand, the set $P$ of primes between $1$ and $N$ clearly forms a multiplicative Sidon subset in $A$ of size at least $(1+o(1)) \frac{N}{\log N},$ by the Prime Number Theorem. Thus, taking $S$ to be the intersection of $P$ with a random ``shift'' of $S_+$, we obtain a bi-Sidon subset in $\{1, \ldots, N\}$ of size at least $c N^{1/2} / \log N$.

    \item Let $A = \{ \gamma, \gamma^2, \ldots, \gamma^N \}$ for an arbitrary $\gamma \neq -1, 0,1$. Then we may consider $S_\times = \{\gamma^s, ~ s \in S_{+}\} \subset A$, where $S_{+} \subset \{1, \ldots, N\}$ is the additive Sidon set described in the previous example. Then $S_\times$ is a multiplicative Sidon subset in $A$ of size $N^{1/2}$. On the other hand, using some elementary number theory, one can upper bound the number of quadruples $\gamma^a+\gamma^b = \gamma^{c} + \gamma^d$ and find a linearly sized additive Sidon subset $B \subset A$. Taking $S$ to be the intersection of $B$ and a random dilate of $S_\times$, we obtain a bi-Sidon subset of size $c N^{1/2}$.
\end{itemize}

These two extreme examples suggest that \emph{any} set $A \subset \R$ should contain a bi-Sidon subset of size $|A|^{1/2+o(1)}$. In this note, we make the first improvement over Ruzsa's above mentioned bound.

\begin{theorem}\label{main}
    For any finite set $A \subset \R$, the size of $S_{bi}(A)$, the largest bi-Sidon subset of $A$, satisfies $|S_{bi}(A)| \gtrsim |A|^{\frac13 + \frac7{78}}.$ 
\end{theorem}

As usual in the subject, here and throughout our note, the asymptotic notation $X \lesssim Y$ expresses that ``$X \le C Y \log^c|A|$ for some constants $C, c \ge 0$.'' 
Our proof, which will be outlined in the next section, uses probabilistic arguments and some energy sum-product estimates. The details will be worked out in the last section.
\smallskip

Our methods are applicable to any other field where appropriate energy sum-product estimates are available. With minor modifications of our arguments, we obtain

\begin{theorem}\label{theorem2}
    Let $A \subset \F$ be a finite subset of a field $\F$, and let $S_{bi}(A)$ denote a largest bi-Sidon subset in $A$. Then we have
    \begin{enumerate}[(i)]
        \item $|S_{bi}(A)| \gtrsim c |A|^{5/12}$ if $\F = \C$;
        \item $|S_{bi}(A)| \gtrsim c |A|^{2/5}$ if $\F = \F_p$ and $|A|\le p^{5/8}$;
        \item $|S_{bi}(A)| \ge c |A|^{1/3+\delta}$ if $\F = \F_p$ and $|A| \le p^{1-\varepsilon},$ \\
        where $\delta =\delta(\varepsilon)>0$ is a constant depending on     $\varepsilon.$
    \end{enumerate}
\end{theorem}

It is an interesting unsolved problem to decide whether (iii) holds with an absolute constant $\delta>0$, without making any assumption on the size of the subset $A\subseteq \F_p$.

\medskip

{\em Acknowledgements.} We thank Imre Ruzsa for helpful conversations and for calling our attention to the problem addressed in this note. We are also grateful to Ilya Shkredov for updating us on the latest literature on energies and sum-product theorems. The first author's work was partially supported by ERC Advanced Grant ``GeoScape'' no.~882971 and NKFIH (National Research, Development and Innovation Office) grant K-131529. The second author was supported by the Jane Street Graduate Fellowship.

\section{Freiman homomorphisms, energies, proof idea}

In this section, we collect some standard definitions and results from additive combinatorics, which will be needed for the proof.

\parag{Freiman homomorphisms.}
Let $G, H$ be abelian groups and let $f: A \rightarrow H$ be a map defined on some subset $A \subset G$. We say that $f$ is a {\em Freiman homomorphism} on $A$ if for any elements $a, b, a', b' \in A$ such that $a+b=a'+b'$ it follows that $f(a)+f(b)=f(a')+f(b')$. A {\em Freiman emdedding} is a Freiman homomorphism $f: A \rightarrow H$ such that $f(a) \neq f(b)$ for any distinct $a\neq b \in A$ (i.e. $f$ is a Freiman isomorphism from $A$ to $f(A)$, see \cite{TaoVu}). 

For the sake of completeness, we include the short proofs of two simple and well known facts on Freiman embeddings that we need.

\begin{prop}\label{prop1}
    Let $G$ be a torsion-free abelian group and $A \subset G$ be a finite subset of $G$. 
    
    Then there exists a Freiman embedding $f: A \rightarrow \Z$.
\end{prop}

\begin{proof}
    We may assume that $G \cong \Z^d$. Let $n_1 < n_2 <\ldots < n_d$ be a sequence of integers and define for $a = (a_1, \ldots, a_d) \in A$: 
    \[
    f(a) = n_1 a_1 + n_2 a_2 + \ldots + n_d a_d \in \Z.
    \]
    Clearly, $f$ is a Freiman homomorphism, and it is an embedding if we choose $n_1, \ldots, n_d$ to be growing fast enough.
\end{proof}

\begin{prop}\label{prop2}
    Let $A \subset \Z$ be a set of size $N$. For any prime $p\ge 2$ and any integer $d \ge 1$ satisfying the condition $N \le 2^{-2d-1}p^d$, there exist a subset $A' \subset A$ of size at least $2^{-d-1}N$ and a Freiman embedding $f: A' \rightarrow \F_p^d$.
\end{prop}

\begin{proof}
    Define $\varphi_p: \R \rightarrow \F_p$ to be the map $\varphi_p(x) = [p x] \pmod p$. Let $\theta_1, \ldots, \theta_d \in [0,1]$ be uniformly random and independent numbers and define a (random) map $f: \Z \rightarrow \F_p^d$ by
    \[
    f(a) = (\varphi_p(\theta_1 a), \ldots, \varphi_p(\theta_d a)).
    \]
    Let $B \subset A$ be the subset of elements $a \in A$ such that $\{p\theta_i a\} \in [0, 1/2)$ for every $i=1, \ldots, d$ (here $\{x\}$ denotes the fractional part of a real number $x$). Observe that $f$ restricted to $B$ is a Freiman homomorphism. Indeed, for each $i=1, \ldots, d,$ observe that
    \[
    \varphi_p(\theta_i b)+\varphi_p(\theta_i b') = [p\theta_1 b]+[p\theta_i b'] = [p\theta_i (b+b')] = \varphi_p(\theta_i(b+b')),
    \]
    for $b, b' \in B$. We also have that $\E |B| = 2^{-d} N,$ since the $\theta_i$ are independent and the fractional parts $\{n \theta_i\}$ are distributed uniformly in $[0,1)$, for any integer $n$.
    
    Let $C \subset A$ be the set of elements $a \in A$ such that there exists $a' \neq a$ with $f(a) = f(a')$. Note that, for fixed distinct elements $a \neq a'\in A$, we have
    \[
    \Pr[ f(a)=f(a') ] = \Pr[\varphi_p(\theta_1 a) = \varphi_p(\theta_1 a')]^d \le \Pr[ \operatorname{dist}(\theta_1 (a-a'), \Z) \le 1/p ]^d \le (2/p)^d.
    \]
    So, by the union bound, the probability that $a \in C$ is bounded by $N (2/p)^d \le 2^{-d-1}$. We conclude that $\E |C| \le 2^{-d-1} N$. Finally, notice that if we let $A' = B\setminus C$, then $f: A'\rightarrow \F_p^d$ is a Freiman embedding and $\E|A'|\ge \E|B|-\E|C| \ge 2^{-d-1}N$. Thus, there is a choice of $\theta_1, \ldots, \theta_d$, for which  $|A'|\ge 2^{-d-1}N$ holds.
\end{proof}

\parag{Additive and multiplicative energies.}
Let $A$ be a finite set in a field $\F$. Using the notation $A^4=A\times A\times A\times A,$ 
define the \emph{additive energy} $E^+(A)$ and the \emph{multiplicative energy} $E^\times(A)$ of $A$ as follows. Let
\[
E^+(A) =|\{(a, b, a', b') \in A^4:~ a+b=a'+b'\}|,
\]
\[
E^\times(A) = |\{(a, b, a', b') \in A^4:~ a b=a' b'\}|.
\]

Improving an earlier result of Rudnev, Shkredov, and Stevens \cite{Rudnev19}, 
Shakan \cite[Theorem 1.11]{Shakan} established the following \emph{``energy sum-product estimate''}.

\begin{theorem}[Shakan]\label{shakan}
Let $A \subset \R$ be a finite set. There exists a subset $A' \subset A$ of size at least $|A|/2$ such that 
$$\min \{E^+(A'), E^\times(A') \} \lesssim |A|^{3 - \delta},$$
where $\delta=\frac7{26}$.
\end{theorem}

Actually, \cite{Shakan} proves a stronger decomposition-type result but we will not need it here.
An example of Balog and Wooley~\cite{BalogWooley} shows that Theorem~\ref{shakan} is false if we replace $\delta$ with any number larger than $\frac23$.

\parag{Proof idea.} To prove Theorem~\ref{main}, we apply Theorem~\ref{shakan} to the set $A\subset\R$, to construct a subset $A'\subset A$ with the required properties. We distinguish two cases, depending on whether the additive energy of $A'$ is smaller or its multiplicative energy. 

Suppose that $A'$ has small multiplicative energy, the other case is analogous. We construct a random \emph{additive} Sidon subset $B \subset A'$ of size $\sim |A|^{1/2}$ which has the `expected' \emph{multiplicative} energy, i.e., each multiplicative quadruple $(a,b,a',b')\in A'^4$ with $ab=a'b'$ has the same probability of appearing in $B^4$ as if $B$ was a uniformly selected random subset in $A'$ of size $|A|^{1/2}$. By taking a random subset of $B$ and deleting an element from each multiplicative quadruple, we obtain a large bi-Sidon subset $S \subset A$. Here the gain is due to the fact that $A'$ had very few multiplicative quadruples. This property was inherited by $B$, and so the deletion argument is more efficient. 
\smallskip

The crucial step of the above argument is the construction of the `sufficiently random' Sidon subset $B\subset A'$. This is achieved by taking a uniformly random parabola $P \subset \F_p^2$ and lifting it to a subset of $A'$ via an appropriate Freiman homomorphism. A similar construction, based on a uniformly random parabola, was recently (and independently from this work) used by Tao \cite{Tao} to provide a counterexample to an unrelated question of Erd\H os.

\section{Proof of Theorem \ref{main}} 

Let $A \subset \R$ be a finite set. Replacing $A$ with the bigger of the sets $A \cap \R_{>0}$ or $(-A) \cap \R_{>0}$, we may assume that $A \subset \R_{>0}$. By Theorem~\ref{shakan}, there exists $A' \subset A$ with $|A'|\ge|A|/2$ such that 
\begin{equation}\label{eq1}
\min \{E^+(A'), E^\times(A') \} \le K |A|^{3 - \delta} ,    
\end{equation}
where $\delta = \frac7{26}$ and $K = C \log^c |A|$ is a logarithmic loss (and $C, c \ge 0$ are absolute constants).
\smallskip

We distinguish two cases depending on which of the two energies in (\ref{eq1}) is smaller. Suppose first that we have $E^\times(A')\le K |A|^{3-\delta}$. Let $p$ be a prime such that $p \in \left[8|A|^{1/2}, 16|A|^{1/2}\right]$. Combining Propositions \ref{prop1} and \ref{prop2}, we can find a subset $A'' \subset A'$ of size at least $|A'|/8$ and a Freiman embedding $f: A'' \rightarrow \F_p^2$.
\smallskip

Let $P \subset \F_p^2$ be {\em a uniformly} selected \emph{random parabola}. More precisely, let $$P_0=\{(t, t^2), ~ t \in \F_p\}$$ 
be the standard unit parabola (mod $p$), and let $g \in \operatorname{Aff}(\F_p^2)$ be a \emph{uniformly} selected \emph{random invertible affine} map $g:\F_p^2\rightarrow\F_p^2$. Set $P = g(P_0)$.

\begin{obs}\label{obs1}
    Let $v_1, v_2, v_3, v_4 \in \F_p^2$ be distinct vectors. Then we have $$\Pr[\{v_1, v_2, v_3\} \subset P] \le p^{-3},\; \mbox{and}\;
    \Pr[\{v_1, v_2, v_3, v_4\} \subset P] \le 4 p^{-4}.$$
\end{obs}
\begin{proof}
    If a triple of vectors $v_i$ are collinear, then they cannot lie on the same parabola. Hence, the probabilities in question are 0. So, we can assume that no three of the $v_i$'s are collinear.

    Note that the probability $q=\Pr[\{v_1, v_2, v_3\} \subset P]$ does not depend on the choice of the (noncollinear) vectors $v_1, v_2, v_3$.  Indeed, any two noncollinear triples $(v_1, v_2, v_3)$ and $(u_1, u_2, u_3)$ can be mapped into each other by an affine map $h$. Since the distribution of $P$ is $h$-invariant, the corresponding probabilities are the same. The number of noncollinear triples $(v_1, v_2, v_3)$ in $\F_p^2$ is $p^2(p^2-1)(p^2-p)$, and every parabola $P$ contains exactly $p(p-1)(p-2)$ of them. Therefore, by a double-counting argument we obtain that
    \[
    q = \frac{p(p-1)(p-2)}{p^2(p^2-1)(p^2-p)} = \frac{p-2}{p^2(p+1)(p-1)}<p^{-3},
    \]
    which proves the first inequality.
    
    To verify the second one, we upper bound the conditional probability
    \[
    q' = \Pr[v_4 \in P~|~ v_1, v_2, v_3 \in P].
    \]
    This probability is invariant under affine transformations of the vectors $v_i$, so, applying an appropriate affine map, we may assume that $v_1 = (0, 0)$, $v_2 = (1, 0)$, $v_3 = (0, 1)$, and $v_4 = (x_0, y_0)$ for some $x_0, y_0 \neq 0$ and $x_0+y_0\neq 1$. (Here, we used the assumption that no three $v_i$'s are collinear). Note that there are exactly $p-2$ parabolas $P$ that pass through the points $(0, 0)$, $(1, 0),$ and $(0, 1)$, and that every point $(x_0, y_0)$ belongs to at most 2 of them. One way of seeing this is to write down an explicit parameterization of this family of parabolas, using coordinates. Namely, for any $s \in \F_p\setminus \{0, 1\}$, let
    \[
    P_s = \{ (x, y) \in \F_p^2:~ (x+sy)^2 = x+ s^2 y \}.
    \]
    This implies that $q' \le \frac{2}{p-2}$, and we have
    \[ \Pr[\{v_1, v_2, v_3, v_4\} \subset P] \le q'q \le \frac{2}{p^2(p^2-1)} \le 4p^{-4}\] as required.
\end{proof}

Define $B = f^{-1}(P)$. That is, let $B$ be the set of elements $a \in A''$ that are getting mapped to the random parabola $P$ by the Freiman embedding $f$ obtained at the beginning of the proof. Since $P \subset \F_p^2$ is a Sidon set, $B$ is an \emph{additive} Sidon subset in $A$. Obviously, we have $$\E|B| = p^{-1} |A''| \ge p^{-1} |A'|/8\ge   2^{-7}|A|^{1/2}.$$ 

Next, we estimate the expectation of the \emph{multiplicative} energy of $B$. For convenience, let $E_0^\times(B)$ denote the number of multiplicative quadruples $b_1 b_2= b_3 b_4$ with all $b_i$ distinct, and let $E_1^\times(B)$ the number of quadruples $b_1 b_2= b_3 b_4$ with $b_1 =b_2$ or $b_3=b_4$ but not both.

Let $a_1 a_2 = a_3 a_4$ be a multiplicative quadruple in $A''$.
\begin{itemize}
    \item If the elements $a_1, a_2, a_3, a_4$ are all distinct, then, by Observation \ref{obs1} applied to vectors $v_i = f(a_i)$, we obtain $$\Pr[\{a_1, a_2, a_3, a_4\} \subset B] \le 4 p^{-4}.$$
    \item If $a_1\neq a_2$ and $a_3= a_4$, then, by Observation \ref{obs1}, we obtain $$\Pr[\{a_1, a_2, a_3, a_4\} \subset B] \le p^{-3}.$$ The case $a_1=a_2$ and $a_3\neq a_4$ is similar.
\end{itemize}
\smallskip

By the linearity of expectation, we get
\[
\E[ E_0^\times(B)] \le 4p^{-4} E^\times(A'') \le 4K p^{-4} |A|^{3-\delta} \le K |A|^{1-\delta},
\]
\[
\E[ E_1^\times(B)] \le p^{-3} |A''|^2 \le |A|^{1/2}.
\]

Let $q = c|A|^{\frac{\delta}3 - \frac16}$ for a sufficiently small $c>0$
and let $\tilde B \subset B$ be a random subset of $B$, where each element is included independently with probability $q$. Then we have
\[
\E|\tilde B| = q \E |B| \ge 2^{-7}c |A|^{\frac13+\frac{\delta}{3}},
\]
\[
\E[E_0^\times(\tilde B)] = q^4 \E[E_0^\times(B)] \le c^4 K |A|^{\frac13+\frac{\delta}3}, 
\]
\[
\E[E_1^\times(\tilde B)] = q^3 \E[E_1^\times(B)] \le c^3 |A|^{\delta}.
\]
Note that $\delta \le \frac13+\frac{\delta}3$. 
\smallskip

Finally, let $S \subset B$ be a set obtained from $B$ by removing one element from each nontrivial multiplicative quadruple. It is clear that then $S$ is a \emph{bi-Sidon} set and its expected size is at least 
\[
\E|S| \ge \E|\tilde B| - \E[E_0^\times(\tilde B)] -\E[E_1^\times(\tilde B)] \ge (2^{-7}c - c^4 K-c^3) |A|^{\frac13+\frac{\delta}3}.
\]
Thus, taking $c = 2^{-10} K^{-1}$ and choosing an $S$ with $|S|\ge \E|S|$, gives us the desired bi-Sidon subset $S\subset A$ of size $\kappa |A|^{\frac13+\frac{\delta}3} = \kappa |A|^{\frac{33}{78}}$ where $\kappa \ge c' \log^{-C} |A| $ for some constants $c',C >0$. 
\smallskip

This concludes the proof Theorem~\ref{main} in the case when the second term of (\ref{eq1}) is smaller than the first. In the complementary case, we have that $E^+(A') \le K |A|^{3-\delta}$, and a totally symmetric argument applies\footnote{According to our assumption, $A \subset \R_{>0}$. Since $G = (\R_{>0}, \times)$ is a torsion-free abelian group, we can apply Proposition \ref{prop1} to the pair $(A,G)$).}.   \qed
\bigskip

\parag{Modifications for the Proof of Theorem~\ref{theorem2}.}
The proof closely follows the above arguments. The only difference is that at the proof of (i) and (ii), instead of Shakan's theorem (Theorem~\ref{shakan}), we need to apply the energy sum-product estimate of Rudnev, Shkredov, and Stevens~\cite[Theorem 5]{Rudnev19}.

\begin{theorem}[Rudnev, Shkredov, Stevens]
Let $A$ be a finite subset of a field $\F$. There exists a subset $A' \subset A$ of size at least $|A|/2$ such that 
$$\min \{E^+(A'), E^\times(A') \} \lesssim |A|^{3 - \delta},$$
where $\delta=\frac14$ if $\F=\C$, and $\delta=\frac15$ if $\F=\F_p$ for a prime $p$. In the latter case, we need the additional assumption that $|A|\le p^{\frac58}$.
\end{theorem}

For the proof of (iii), we need to apply a somewhat weaker but more general bound of Balog and Wooley \cite[Theorem 1.3]{BalogWooley}.

\footnotesize{

\bigskip


\noindent\textsc{J. Pach, HUN-REN R\'enyi Institute of Mathematics, Budapest, Re\'altanoda u. 13--15, Budapest 1053, HUNGARY}. {\it{Email address}}: \href{mailto:pach@renyi.hu}{\nolinkurl{pach@cims.nyu.edu}}

\medskip

\noindent\textsc{D. Zakharov, Department of Mathematics, Massachusetts Institute of Technology, Cambridge, MA 02139, USA}. {\it{Email address}}: \href{mailto:zakhdm@mit.edu}{\nolinkurl{zakhdm@mit.edu}}
}

\end{document}